\documentclass{amsart}
\usepackage{amsthm,latexsym,amsfonts,amsmath,amssymb,mathrsfs,array,manfnt,stmaryrd}
\usepackage[pdftex]{graphicx}
\usepackage[all]{xy}
\usepackage{paralist}
\usepackage{framed}
\usepackage{cancel}
\usepackage{enumitem}
\usepackage{asymptote}
\usepackage{units}
\usepackage{setspace}
\usepackage{amscd}
\usepackage{microtype}
\usepackage{tikz}
\usetikzlibrary{matrix}
\usetikzlibrary{calc}
\usepackage{tikz-cd}

\newcommand{\C}{\ensuremath{\mathbb{C}}}
\newcommand{\Q}{\ensuremath{\mathbb{Q}}}

\newcommand{\M}{\mathcal{M}}
\newcommand{\Mbar}{\overline{M}}
\newcommand{\Hbar}{\overline{\mathcal{H}}}

\renewcommand{\P}{\ensuremath{\mathbb{P}}}

\DeclareMathOperator{\Per}{Per}

\usepackage{color}

\newcommand{\margincolor}{red}      
\definecolor{darkgreen}{rgb}{0,0.7,0}

\newcounter{margincounter}
\newcommand{\marginnum}{
\ifnum\value{margincounter}<10
\textcolor{\margincolor}{\begin{picture}(0,0)\put(2.2,2.4){\circle{9}}\end{picture}\footnotesize\arabic{margincounter}}
\else\ifnum\value{margincounter}<100
\textcolor{\margincolor}{\begin{picture}(0,0)\put(4.256,2.5){\circle{11}}\end{picture}\footnotesize\arabic{margincounter}}
\else
\textcolor{\margincolor}{\begin{picture}(0,0)\put(6.8,2.5){\circle{14}}\end{picture}\footnotesize\arabic{margincounter}}
\fi\fi
}

\newcommand{\ba}{\mathbf{a}}
\newcommand{\bb}{\mathbf{b}}
\newcommand{\bbf}{\mathbf{f}}
\renewcommand{\AA}{\mathbb{A}}

\newcommand{\CC}{\mathbb{C}}

\newcommand{\PP}{\mathbb{P}}
\newcommand{\QQ}{\mathbb{Q}}

\newcommand{\cC}{\mathcal{C}}
\newcommand{\cD}{\mathcal{D}}
\newcommand{\cE}{\mathcal{E}}

\newcommand{\cH}{\mathcal{H}}

\newcommand{\cM}{\mathcal{M}}

\theoremstyle{plain}
\newtheorem{theorem}{Theorem}
\numberwithin{theorem}{section}

\newtheorem{thm}[theorem]{Theorem}

\newtheorem*{thm*}{Theorem}

\newtheorem{cor}[theorem]{Corollary}

\newtheorem{lemma}[theorem]{Lemma}

\theoremstyle{definition}

\theoremstyle{remark}

\newtheorem{rem}[theorem]{Remark}

\newtheorem{remark}[theorem]{Remark}

\usepackage{tikz}
\usetikzlibrary{matrix}
\usepackage{fullpage}
\usepackage{bbm}

\usepackage{hyperref}

\begin{document}

\title{Moduli spaces of quadratic maps: arithmetic and geometry}
\author{Rohini Ramadas}
\email{rohini.ramadas@warwick.ac.uk}
\address{Warwick Mathematics Institute, University of Warwick, Coventry, UK}
\date{}
\begin{abstract}
We establish an implication between two long-standing open problems in complex dynamics. The roots of the $n$-th Gleason polynomial $G_n\in\QQ[c]$ comprise the $0$-dimensional moduli space of quadratic polynomials with an $n$-periodic critical point. $\Per_n(0)$ is the $1$-dimensional moduli space of quadratic rational maps on $\P^1$ with an $n$-periodic critical point. We show that if $G_n$ is irreducible over $\QQ$, then $\Per_n(0)$ is irreducible over $\CC$. To do this, we exhibit a $\QQ$-rational smooth point on a projective completion of $\Per_n(0)$, using the admissible covers completion of a Hurwitz space. In contrast, the Uniform Boundedness Conjecture in arithmetic dynamics would imply that for sufficiently large $n$, $\Per_n(0)$ itself has no $\QQ$-rational points.
\end{abstract}
\maketitle

\section{Introduction}
The \emph{$n$th Gleason polynomial} $G_n\in\QQ[c]$ is the polynomial whose roots are the set of $c\in\CC$ such that, under $f_c(z)=z^2+c$, the critical point $0$ has a periodic orbit of exact period $n$. A long-standing open question in complex dynamics asks whether $G_n$ is irreducible over $\QQ$ for all $n$, see \cite{Milnor2014,buff2018postcritically,  goksel2020orbit,buff2021factoring}.

The roots of $G_n$ comprise the $0$-dimensional moduli space of quadratic polynomials having an $n$-periodic critical point. Milnor \cite{Milnor1993} initiated the study of a natural generalization: the curve $\Per_n(0)$ parametrizing quadratic rational maps with an $n$-periodic critical point. Another long-standing open question asks whether $\Per_n(0)$ is irreducible over $\CC$ for all $n$.

\begin{theorem}\label{thm:gleasontopern}
Fix $n\ge 4$. If $G_n$ is irreducible over $\QQ$, then $\Per_n(0)$ is irreducible over $\CC$.
\end{theorem}

Irreducibility of $G_n$ has been shown for $n\le 19$ \cite{DoyleFiliTobinGleason} using \textit{Magma}. We conclude:

\begin{cor}
For $n\le 19$, $\Per_n(0)$ is irreducible over $\CC$.
\end{cor}

The strategy is as follows. Milnor \cite{Milnor1993} shows that every irreducible component of $\Per_n(0)$ contains a polynomial, i.e. a root of $G_n$. This implies (see Corollary \ref{thm:gleasonimpliesperirreducibleoverQ}) that if $G_n$ is irreducible over $\Q$, then so is $\Per_n(0)$. In order to upgrade irreducibility over $\QQ$ to irreducibility over $\CC$, we show:

\begin{thm}\label{prop:Qrationalsmoothpoint}
For every $n\ge 4$, there exists a projective completion of $\Per_n(0)$ that has a smooth $\Q$-rational point $f_{\star}$.
\end{thm}

A $\QQ$-rational point of a $\QQ$-irreducible variety must lie on every $\CC$-irreducible component (because the Galois group acts transitively on the set of $\CC$-irreducible components, but fixes $\QQ$-rational points). On the other hand, a smooth point can lie on at most one $\CC$-irreducible component. We conclude that if $\Per_n(0)$ is irreducible over $\QQ$, then it is irreducible over $\CC$.

\subsection{Finding the smooth \texorpdfstring{$\QQ$}{Q}-rational point at infinity in a Hurwitz space} The strategy of upgrading $\QQ$-irreducibility to $\CC$-irreducibility by finding a smooth $\QQ$-rational point has been fruitful in dynamics \cite{bousch1992quelques, BuffEpsteinKoch2018}. However, in our setting, there is a major obstacle to finding a smooth $\QQ$-rational point, as follows. The Uniform Boundedness Conjecture for arithmetic dynamics would imply that for large $n$, $\Per_n(0)$ has no $\QQ$-rational point \cite{mortonsilverman1994}  --- indeed, $\Per_5(0)$ has no $\QQ$-rational point \cite{RamadasSilversmith2}. Milnor \cite{Milnor1993} originally defined $\Per_n(0)$ as a subset of the moduli space $\cM_2\cong \CC^2$ parametrizing quadratic rational maps up to conjugacy, and studied the closure of $\Per_n(0)$ in $\CC\PP^2\supset \CC^2$. The intersection of this closure with the line at infinity in $\CC\PP^2$ contains $\QQ$-rational points, but they are very singular in general \cite{Stimson1993, Milnor1993}.

We circumvent this obstacle by instead considering a birational model $\cE_n$ of $\Per_n(0)$ that embeds naturally in a Hurwitz space $\cH_n$ (\cite{Epstein2009, HironakaKoch2017}, see Section \ref{sec:moduli}). The Hurwitz space $\cH_n$ has a completion $\Hbar_n$ whose points correspond to \textit{admissible covers}, which are branched coverings of nodal curves \cite{HarrisMumford1982, AbramovichCortiVistoli2003}. We apply techniques developed in \cite{RamadasSilversmith2} to study the completion $\overline{\cE_n}$ of $\cE_n$ in $\Hbar_n$, and find a smooth $\QQ$-rational point $f_{\star}\in\overline{\cE_n}\setminus\cE_n$. As $f_{\star}$ is an admissible cover, it admits an interpretation as a ``degenerate quadratic map" (see Figure \ref{fig:fstar} and Section \ref{sec:fstar}). The primary novel aspect of the present paper is the use of this modular interpretation to establish both $\QQ$-rationality and smoothness of $f_{\star}$. This is a proof-of-concept for the usefulness of moduli spaces of admissible covers in studying dynamics on $\PP^1$.

\begin{rem}
This technique appears to generalize well. We have found two other smooth $\QQ$-rational points on $\overline{\cE_n}$, as well as a smooth $\QQ$-rational point on an analogous completion of the space of quadratic rational maps on $\P^1$ with a pre-periodic critical point of pre-period $2$ (and arbitrary period).
\end{rem}

\subsection{Notes and references} The moduli space $\Per_n(0)$ is an example of a \textit{critical-orbit-relation} space of rational maps, i.e. it parametrizes conjugacy classes of rational maps of fixed degree, with a prescribed condition on a subset of critical orbits. Critical-orbit-relation spaces have been conjectured by Baker-DeMarco \cite{BakerDeMarco2013} to be the only families of rational maps containing a Zariski-dense subset of post-critically finite rational maps; this conjecture has been proved in special cases \cite{BakerDeMarco2013, FavreGauthier2015, DeMarcoWangYe2015}. 

In general, very little is known about the geometry of critical-orbit-relation spaces, but irreducibility has been proved in some important infinite families. Arfeux-Kiwi \cite{arfeux2020irreducibility} proved that the analog of $\Per_n(0)$ in the space of cubic polynomials is irreducible over $\CC$, and Buff-Epstein-Koch \cite{BuffEpsteinKoch2018} proved irreducibility over $\CC$ of spaces of quadratic rational maps (as well as cubic polynomials) with a pre-fixed critical point.

Finally, we note that there are natural transcendental covering spaces of $\Per_n(0)$ (more precisely, of $\cE_n$) called \textit{deformation spaces} \cite{Epstein2009, Rees2009}, whose connectedness and contractibility have been studied in \cite{HironakaKoch2017, Hironaka2019, FirsovaKahnSelinger2016}.

\subsection{Acknowledgements} This project was carried out at the Spring 2022 special semester in Complex Dynamics at the Mathematical Sciences Research Institute (MSRI). I am grateful to MSRI, the program organizers, and the program members for providing a vibrant research environment. I'm especially grateful to Curtis McMullen for bringing Theorem \ref{thm:everycomponentofperncontainsagleasonroot} (\cite{Milnor1993}) to my attention, to Laurent Bartholdi for explaining Milnor's argument, and to both for useful comments on this manuscript. I'm also grateful to Rob Silversmith for useful conversations, typesetting the figures, and for useful feedback on this manuscript.

\begin{figure}
    \centering
    \begin{tikzpicture}[scale=1.5]
   \foreach \x in {1,2,3,5.5,6.5,7.5} {
     \draw (\x+.5,-1) circle(.5);
     \draw (\x+.5,1.5) ellipse(.5 and .2);
     \draw (\x+.5,3) ellipse(.5 and .2);
   }
   \draw (4.75,-1) node {$\cdots$};
   \draw (4.75,1.5) node {$\cdots$};
   \draw (4.75,3) node {$\cdots$};
   \draw (.5,-1) circle(.5);
   \draw (.5,2.25) ellipse(.66 and 1.12);
   \draw[->,thick] (4,.6)--(4,-.2);
   \draw (4,.2) node[right] {$f_\star$};
   \draw (1,-.5) node {$\theta_1$};
   \draw (2,-.5) node {$\theta_n$};
   \draw (3,-.5) node {$\theta_{n-1}$};
   \draw (6.5,-.5) node {$\theta_5$};
   \draw (7.5,-.5) node {$\theta_4$};
   \draw (1,1.2) node {$\eta_n$};
   \draw (2,1.2) node {$\eta_{n-1}$};
   \draw (3,1.2) node {$\eta_{n-2}$};
   \draw (6.5,1.2) node {$\eta_4$};
   \draw (7.5,1.2) node {$\eta_3$};
   \draw (1,3.3) node {$\eta_n'$};
   \draw (2,3.3) node {$\eta_{n-1}'$};
   \draw (3,3.3) node {$\eta_{n-2}'$};
   \draw (6.5,3.3) node {$\eta_4'$};
   \draw (7.5,3.3) node {$\eta_3'$};
   \draw (1.5,-1.7) node {$D_1$};
   \draw (2.5,-1.7) node {$D_n$};
   \draw (3.5,-1.7) node {$D_{n-1}$};
   \draw (6,-1.7) node {$D_6$};
   \draw (7,-1.7) node {$D_5$};
   \draw (8,-1.7) node {$D_4$};
   \draw (.5,-1.7) node {$D_2$};
   \draw (1.5,1.9) node {$C_n$};
   \draw (2.5,1.9) node {$C_{n-1}$};
   \draw (3.5,1.9) node {$C_{n-2}$};
   \draw (6,1.9) node {$C_5$};
   \draw (7,1.9) node {$C_4$};
   \draw (8,1.9) node {$C_3$};
   \draw (-.4,2.25) node {$C_1$};
   \draw (1.5,2.6) node {$C_n'$};
   \draw (2.5,2.6) node {$C_{n-1}'$};
   \draw (3.5,2.6) node {$C_{n-2}'$};
   \draw (6,2.6) node {$C_5'$};
   \draw (7,2.6) node {$C_4'$};
   \draw (8,2.6) node {$C_3'$};
   \draw (1.5,-1) node {$b_1$};
   \draw (2.5,-1) node {$b_n$};
   \draw (3.5,-1) node {$b_{n-1}$};
   \draw (6,-1) node {$b_6$};
   \draw (7,-1) node {$b_5$};
   \draw (8.25,-1) node {$b_3$};
   \draw (7.75,-1) node {$b_4$};
   \draw (.25,-1) node {$b_2$};
   \draw (.75,-1) node {$b_*$};
   \draw (1.5,1.5) node {$a_n$};
   \draw (2.5,1.5) node {$a_{n-1}$};
   \draw (3.5,1.5) node {$a_{n-2}$};
   \draw (6,1.5) node {$a_5$};
   \draw (7,1.5) node {$a_4$};
   \draw[gray] (8.25,1.5) node {$a_2'$};
   \draw (7.75,1.5) node {$a_3$};
   \draw[gray] (1.5,3) node {$a_n'$};
   \draw[gray] (2.5,3) node {$a_{n-1}'$};
   \draw[gray] (3.5,3) node {$a_{n-2}'$};
   \draw[gray] (6,3) node {$a_5'$};
   \draw[gray] (7,3) node {$a_4'$};
   \draw (8.25,3) node {$a_2$};
   \draw[gray] (7.75,3) node {$a_3'$};
   \draw (.25,2.25) node {$a_1$};
   \draw (.75,2.25) node {$a_*$};
   \draw (9.5,2.25) node {$C_\star$};
   \draw (9.5,-1) node {$D_\star$};
 \end{tikzpicture}
    \caption{The point $f_{\star}\in \Hbar_n$, see Section \ref{sec:fstar}.  On $C_1\cong\PP^1_{\QQ}$, coordinates are chosen so that $\eta_n$ is at ``$1$", and $\eta_{n}'$ at ``$-1$",  $a_*$ at ``$\infty$" and $a_1$ at ``$0$".}
    \label{fig:fstar}
\end{figure}
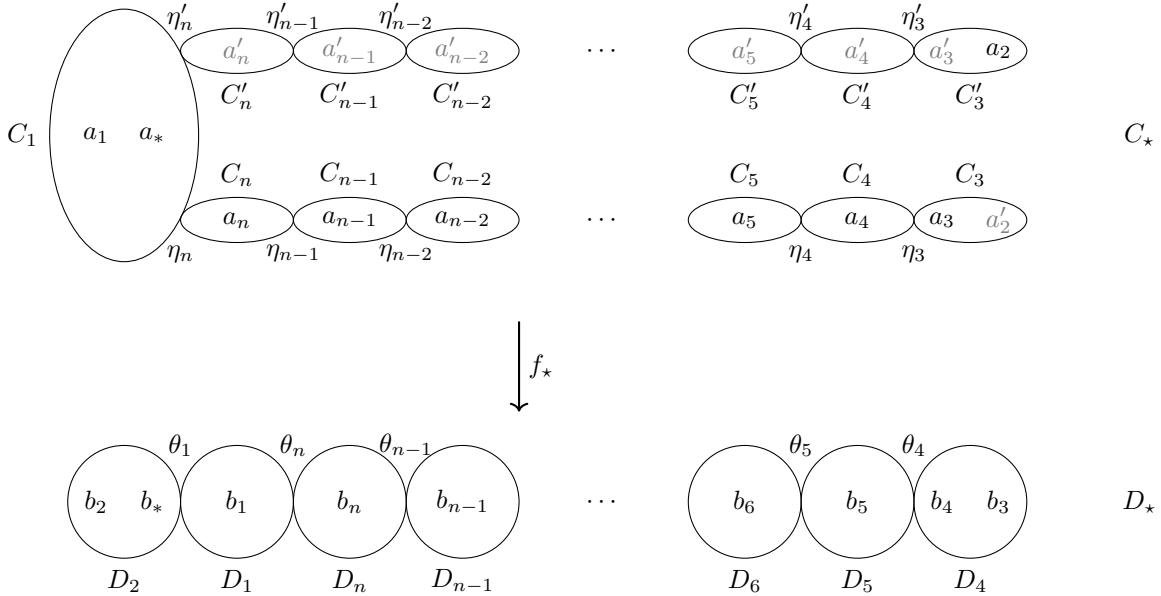

\section{Background: the moduli spaces}\label{sec:moduli}

\subsection{Quadratic rational maps} We denote by $\cM_2$ the space of quadratic rational self-maps of $\P^1$, up to conjugacy. In fact, $\cM_2$ is isomorphic to $\AA^2$ over $\C$ \cite{Milnor1993} and over $\Q$ \cite{Silverman1998}. For $n \ge 1$, let $\Per_n(0)\subsetneq \cM_2$ be the one-dimensional $\QQ$-subvariety parametrizing maps with an $n$-periodic critical point. $\Per_1(0)$ is the locus of quadratic polynomials, and is a (vertical) line in $\cM_2\cong \AA^2$. For $n\ge 2$, under the natural identification of $\Per_1(0)$ with the family $\{f_c(z)=z^2+c\}$, $\Per_n(0)\cap \Per_1(0)$ is the set of roots of  $G_n$. Lemma 4.1 of \cite{Milnor1993} establishes that for $n\ge 2$, $\Per_n(0)$ only intersects the line at infinity in $\P^2$ at points whose coordinates are roots of unity. In particular, for $n\ge 2$, $\Per_n(0)$ does not meet $\Per_1(0)$ at the line at infinity, which implies:

\begin{theorem}\label{thm:everycomponentofperncontainsagleasonroot} [Milnor; \cite{Milnor1993}; Lemma 4.1 and Theorem 4.2]
Every irreducible component of $\Per_n(0)$ has non-empty intersection with $\Per_1(0)$.
\end{theorem} 

As an immediate consequence, we obtain:

\begin{cor}\label{thm:gleasonimpliesperirreducibleoverQ}
If $G_n$ is irreducible over $\Q$, then so is $\Per_n(0)$.  
\end{cor}
\begin{proof}
 We apply Theorem \ref{thm:everycomponentofperncontainsagleasonroot} to note that a non-trivial factorization of the equation cutting out $\Per_n(0)$ in $\M_2$ would restrict a non-trivial factorization of $G_n$ on $\Per_1(0)$.
\end{proof}

\subsection{The Hurwitz space} Let $M_{0,n}$ be the moduli space of $n$-pointed genus-$0$ curves, i.e., $M_{0,n}$ parametrizes tuples $(C,p_1,\ldots, p_n)$, where $C$ is a smooth genus-$0$ curve, and $p_1,\ldots, p_n\in C$ are distinct. Let $\cH_n$ be the space parametrizing tuples $$(C, D, f, a_*, a_1,\ldots, a_n, a_2',\ldots, a_n', b_*, b_1,\ldots, b_b),$$ where $C$ and $D$ are smooth genus-$0$ curves, $a_*,a_1,\ldots,a_n, a_2',\ldots, a_n'\in C$ are distinct, $b_*,b_1,\ldots,b_n\in D$ are distinct, and $f:C\to D$ is a degree-$2$ map for which $a_*$ and $a_1$ are critical and such that $f(a_*)=b_*$, $f(a_1)=b_2$, $f(a_n)=f(a_n')=b_1$, and $f(a_i)=f(a_i')=b_{i+1}$ for $i=2,\ldots, n-1$. There are two maps $\pi_a, \pi_b:\cH_n\to \M_{0,n}$, where
\begin{align*}
    &\pi_a((C, D, f, a_*, a_1,\ldots, a_n, a_2',\ldots, a_n', b_*, b_1,\ldots, b_b))=(C,a_1,\ldots, a_n)\\
    &\pi_b((C, D, f, a_*, a_1,\ldots, a_n, a_2',\ldots, a_n', b_*, b_1,\ldots, b_b))=(D,b_1,\ldots, b_n)
\end{align*}

Let $\Delta\subseteq M_{0,n}\times M_{0,n}$ be the diagonal, and $\cE_n:=(\pi_a\times \pi_b)^{-1}(\Delta)$.

\begin{lemma}[See also \cite{HironakaKoch2017, RamadasSilversmith2}]\label{lem:birtaional}
There is a birational map $\nu:\cE_n\to\Per_n(0)$, defined over $\QQ$.
\end{lemma}

\begin{proof}
There are universal curves $\cC$ and $\cD$ over $\cH_n$, together with sections $\ba_*,\ba_1,\ldots,\ba_n, \ba_2',\ldots, \ba_n'$ of $\cC$, and sections $\bb_*,\bb_1,\ldots,\bb_n$ of $\cD$, and a universal degree-$2$ map $\bbf:\cC\to\cD$, all defined over $\QQ$ \cite{AbramovichCortiVistoli2003}. Because $\cC$ and $\cD$ have more than $3$ disjoint sections, they are each isomorphic to $\PP^1\times\cH_n$. There is an isomorphism $\cC|_{\cE_n}\to\cD|_{\cE_n}$, defined over $\QQ$, identifying $\ba_i$ with $\bb_i$. Under this identification, $\bbf$ restricts to a family of quadratic rational self-maps of $\PP^1$ parametrized by $\cE_n$. This induces a morphism $\nu:\cE_n\to\M_2$, defined over $\QQ$ \cite{Silverman1998}. Since every rational map in the family has an $n$-periodic critical point, we have $\nu(\cE_n)\subset\Per_n(0)$. To see that $\nu$ is birational, we construct a rational inverse. Let $\Per_n(0)^{\circ}\subsetneq\Per_n(0)$ be the non-empty Zariski-open subset (defined over $\QQ$) where one critical point is not $n$-periodic and the map has no nontrivial automorphisms. On $\Per_n(0)^{\circ}$, one can mark the $n$-periodic critical point and its forward image, the inverse images of the points in the $n$-cycle, and the other critical point and its image. This induces a map $\Per_n(0)^{\circ}\to\cH_n$, which by construction is an inverse of $\nu$. 
\end{proof}

\begin{cor}\label{cor:PerntoEn}
$\Per_n(0)$ is irreducible over $\QQ$ (resp. $\CC$) if and only if $\cE_n$ is irreducible over $\QQ$ (resp.  $\CC$).
\end{cor}

\subsection{Stable curves and admissible covers} Given a finite set $S$, an \textit{$S$-marked stable genus-$0$ curve} is a connected nodal genus-$0$ curve $C$, together with an injection from $S$ into the smooth locus of $C$ such that every irreducible component has at least three points that are either marked or nodes.  $\Mbar_{0,n}$ is the Deligne-Mumford completion of $M_{0,n}$; its geometric points correspond to $n$-marked stable genus-$0$ curves. Let $\Hbar_n$ be the projective completion of $\cH_n$ whose geometric points correspond to \textit{admissible covers} \cite{AbramovichCortiVistoli2003, HarrisMumford1982}, i.e. tuples $$(C, D, f, a_*, a_1,\ldots, a_n, a_2',\ldots, a_n', b_*, b_1,\ldots, b_b),$$ where 
\begin{itemize}
    \item $C$ is a stable $\{a_*,a_1,\ldots,a_n, a_2',\ldots, a_n'\}$-marked genus-$0$ curve,
    \item $D$ is a stable $\{b_*,b_1,\ldots,b_n\}$-marked genus-$0$ curve,
    \item $f:C\to D$ is a finite degree-$2$ map for which $a_*$ and $a_1$ are critical, satisfying:
    \begin{itemize}
    \item $f(a_*)=b_*$, $f(a_1)=b_2$, and $f(a_n)=f(a_n')=b_1$, for $i=2,\ldots, n-1$, $f(a_i)=f(a_i')=b_{i+1}$,
\item nodes of $C$ map to nodes of $D$, and smooth points of $C$ map to smooth points of $D$, 
\item away from $a_*$ and $a_1$, the only ramification of $f$ is at nodes of $C$, and 
\item (Balancing condition) at each node $\eta\in C$, the two different branches at $\eta$ map to the two different branches at $f(\eta)\in D$, and map with equal local degree.
\end{itemize}
    
\end{itemize}

The maps $\pi_a$ and $\pi_b$ extend to maps from $\Hbar_n$ to $\Mbar_{0,n}$. Let $\overline{\Delta}\subseteq \Mbar_{0,n}\times \Mbar_{0,n}$ be the diagonal, and let $\overline{\cE_n}^{+}:=(\pi_a\times \pi_b)^{-1}(\overline{\Delta})$. Note that $\overline{\cE_n}^{+}$ contains the Zariski closure $\overline{\cE_n}$ of $\cE_n$, but in general also contains irreducible components supported on $\Hbar_n\setminus\cH_n$.

\section{The smooth \texorpdfstring{$\QQ$}{Q}-rational point} \label{sec:fstar}

In this section, we fix $n\ge 4$. We describe a $\QQ$-rational point $(C_{\star}, D_{\star},f_{\star})\in\Hbar_n$, as depicted in Figure \ref{fig:fstar}. Precisely:
\begin{itemize}
    \item The $C_i$s, $C_i'$s, and $D_i$s label irreducible components of $C_{\star}$ and $D_{\star}$ (depicted as ellipses), 
    \item The $a_i$s, $a_i'$s, and $b_i$s inside an ellipse are the marked points on that irreducible component,
    \item $f_{\star}$ maps $C_{\star}$ to $D_{\star}$ ``vertically", and the restriction of $f_{\star}$ to any $C_i$ or $C_i'$ \textit{except} for $C_1$ is the unique isomorphism sending marked points to marked points and nodes to nodes as depicted.
    \item On $C_1\cong\PP^1_{\QQ}$, coordinates are chosen so that $\eta_n$ is at ``$1$", and $\eta_{n}'$ at ``$-1$",  $a_*$ at ``$\infty$" and $a_1$ at ``$0$". The map $f_{\star}|_{C_1}:C_1\to D_2$ is the map $z\mapsto z^2$, with coordinates on $D_2$ so that $b_{*}$ is at $\infty$, $b_2$ is at $0$, and $\theta_1$ is at $1$.
\end{itemize}

\begin{lemma}\label{lem:Qrational}
The point $(C_{\star}, D_{\star},f_{\star})\in\Hbar_n$ is $\QQ$-rational. 
\end{lemma}
\begin{proof}
The above construction was over $\QQ$. In particular, $C_1$ is the only component of $C_{\star}$ with more than three special points, and the cross-ratio of the four special points on $C_1$ is $-1$. Every component of $D_{\star}$ has exactly three special points.
\end{proof}

\begin{lemma}\label{lem:equalizer}
The point $(C_{\star}, D_{\star},f_{\star})$ lies in $\overline{\cE_n}^{+}$. 
\end{lemma}
\begin{proof}

The common stabilization of $(C_{\star}, a_1,\ldots, a_n)$ and $(D_{\star}, b_1,\ldots, b_n)$ is the stable curve $(X_{\star},p_1,\ldots, p_n)$ depicted in Figure \ref{fig:Xstar}, so \begin{align*}
    \pi_a((C_{\star}, D_{\star},f_{\star}))&=\pi_b((C_{\star}, D_{\star},f_{\star}))=(X_{\star},p_1,\ldots, p_n)\in\Mbar_{0,n}.\qedhere
\end{align*}
\end{proof}

\begin{lemma}\label{lem:smoothpoint}
The point $(C_{\star}, D_{\star},f_{\star})$ lies in $\overline{\cE_n}$, and is a smooth point of $\overline{\cE_n}$. 
\end{lemma}
\begin{proof}
 The proof is a computation in local coordinates on $\Hbar_n$, as developed in Sections 3.4 and 3.5 in \cite{RamadasSilversmith2}, based on deformation theory arguments from \cite{DeligneMumford1969, Knudsen1983, HarrisMumford1982}. We base change to $\C$, and use notation from Figures \ref{fig:fstar} and \ref{fig:Xstar}. 

In a (formal) neighbourhood of $(C_{\star}, D_{\star},f_{\star})$, $\Hbar_n$ admits local coordinates $(s_1,\ldots s_{n-2})$, where $s_i$ is a \textit{node-smoothing parameter} for $\eta_{n+1-i}$ -- as well as for $\eta_{n+1-i}'$ and for $f_{\star}(\eta_{n+1-i})$. (This last fact is because all of the nodes $\eta_i$ are unramified.) In a (formal) neighbourhood of $(X_{\star},p_1,\ldots, p_n)$, $\Mbar_{0,n}$ admits local coordinates $t_1,\ldots, t_{n-3}$, where $t_i$ is a node-smoothing parameter for $\gamma_i$. Using Figures \ref{fig:fstar} and \ref{fig:Xstar}, we can write $\pi_a^{*}(t_j)$ and $\pi_b^{*}(t_j)$ in terms of the coordinates $s_i$, as follows. Observe:
\begin{itemize}
    \item On $X_{\star}$, the only node separating $p_1$ and $p_2$ from $p_3,\ldots, p_n$ is $\gamma_1$, for which $t_1$ is a node-smoothing parameter.
    \item On $C_{\star}$, the only node that separates $a_1$ and $a_2$ from $a_3,\ldots, a_n$ is $\eta_n$, for which $s_1$ is a node-smoothing parameter.
    \item On $D_{\star}$, the only node that separates $b_1$ and $b_2$ from $b_3,\ldots, b_n$ is $\theta_n$, for which $s_2$ is a node-smoothing parameter.
\end{itemize}
  We infer that $\pi_a^{*}(t_1)=\alpha_1 s_1$ and $\pi_b^{*}(t_1)=\beta_1 s_2$, where $\alpha_1$ and $\beta_1$ are non-vanishing regular functions on a formal neighbourhood of $f_{\star}\in\Hbar_n$. Similar computations tell us that for all $i=1,2,\ldots n-3$, we have 
\begin{align*}
   \pi_a^{*}(t_i)&=\alpha_i s_i&&\text{and}
   &\pi_b^{*}(t_i)&=\beta_i s_{i+1},
\end{align*}
 where $\alpha_i$ and $\beta_i$ are non-vanishing regular functions on a formal neighbourhood of $f_{\star}$.

In a formal neighbourhood of $f_{\star}$, $\overline{\cE_n}^{+}$ is cut out by the equations $=\pi_a^{*}(t_i)=\pi_b^{*}(t_i)$, which can be rewritten as $\alpha_i s_i=\beta_i s_{i+1}$. (In these coordinates, $f_{\star}$ is identified with the origin.) These equations describe a (germ of a) curve, smooth at the origin, that does not lie in any of the coordinate hypersurfaces $\{s_i=0\}$. On the other hand, in these coordinates, the boundary $\Hbar_n\setminus\cH_n$ is identified with the union of the $(n-2)$ coordinate hypersurfaces $\{s_i=0\}$. We conclude two things: First, near $f_{\star}$, $\overline{\cE_n}^{+}$ does not generically lie in $\Hbar_n\setminus\cH_n$. In other words, $f_{\star}$ is in $\overline{\cE_n}$. Second, $\overline{\cE_n}$ is smooth at $f_{\star}$.  
\end{proof}

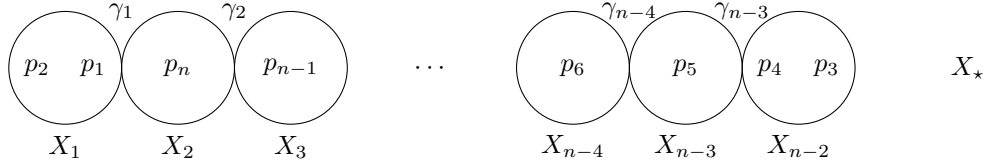
\begin{figure}
    \centering
    \begin{tikzpicture}[scale=1.5]
\foreach \x in {1,2,3,5.5,6.5,7.5} {
\draw (\x+.5,-1) circle(.5);
}
\draw (4.75,-1) node {$\cdots$};
\draw (2,-.5) node {$\gamma_1$};
\draw (3,-.5) node {$\gamma_2$};
\draw (6.5,-.5) node {$\gamma_{n-4}$};
\draw (7.5,-.5) node {$\gamma_{n-3}$};
\draw (1.5,-1.7) node {$X_1$};
\draw (2.5,-1.7) node {$X_2$};
\draw (3.5,-1.7) node {$X_3$};
\draw (6,-1.7) node {$X_{n-4}$};
\draw (7,-1.7) node {$X_{n-3}$};
\draw (8,-1.7) node {$X_{n-2}$};
\draw (1.75,-1) node {$p_1$};
\draw (2.5,-1) node {$p_n$};
\draw (3.5,-1) node {$p_{n-1}$};
\draw (6,-1) node {$p_6$};
\draw (7,-1) node {$p_5$};
\draw (8.25,-1) node {$p_3$};
\draw (7.75,-1) node {$p_4$};
\draw (1.25,-1) node {$p_2$};
\draw (9.5,-1) node {$X_\star$};
\end{tikzpicture}
    \caption{The stable curve $(X_{\star}, p_1,\ldots, p_n)\in\Mbar_{0,n}$ that is the common stabilization of $(C_{\star}, a_1,\ldots, a_n)$ and $(D_{\star}, b_1,\ldots, b_n)$.}
    \label{fig:Xstar}
\end{figure}

We tie together the various strands to complete the proof of Theorem \ref{thm:gleasontopern}.

\begin{proof}[Proof of Theorem \ref{thm:gleasontopern}] 
Suppose that $G_n$ is irreducible over $\QQ$. Then by Theorem \ref{thm:gleasonimpliesperirreducibleoverQ}, $\Per_n(0)$ is irreducible over $\QQ$. By Corollary \ref{cor:PerntoEn}, $\cE_n$ is irreducible over $\QQ$, and therefore so is  $\overline{\cE_n}$. Since $\overline{\cE_n}$ is $\QQ$-irreducible, the Galois group acts transitively on its set of $\CC$-irreducible components. The point $f_{\star}\in \overline{\cE_n}$, being $\QQ$-rational \ref{lem:Qrational}, must therefore lie on every $\CC$-irreducible component. On the other hand, $f_{\star}$ is a smooth point \ref{lem:smoothpoint} and so can lie on at most one $\CC$-irreducible component. We conclude that $\overline{\cE_n}$ is irreducible over $\CC$. By Corollary \ref{cor:PerntoEn}, $\Per_n(0)$ is irreducible over $\CC$.
 \end{proof}

\begin{remark}[Proof of Theorem \ref{prop:Qrationalsmoothpoint}] We have not quite proved Theorem \ref{prop:Qrationalsmoothpoint} --- doing so was not strictly necessary for Theorem \ref{thm:gleasontopern}. $\overline{\cE_n}$ has a smooth $\QQ$-rational point, but it is not a completion of $\Per_n(0)$, merely a birational model. In fact, using notation from the proof of Lemma \ref{lem:birtaional}, $\nu:\cE_n\to\Per_n(0)$ identifies pairs of points where both critical points are $n$-periodic. We prove Theorem \ref{prop:Qrationalsmoothpoint} by glueing $\Per_n(0)$ to an open subset of $\overline{\cE_n}$ along a common open subset. Let $S=\nu^{-1}(\Per_n(0)\setminus\Per_n(0)^{\circ})$; $S$ is a $0$-dimensional $\QQ$-subscheme of $\cE_n$ and therefore also of $\overline{\cE_n}$. Let $\overline{\cE_n}^{\circ}:=\overline{\cE_n}\setminus S$, and note that $\cE_n\setminus S$ is open in $\overline{\cE_n}^{\circ}$. By the proof of Lemma \ref{lem:birtaional}, the restriction $\nu:\cE_n\setminus S\to \Per_n(0)^{\circ}$ is an isomorphism. We glue $\overline{\cE_n}^{\circ}$ with $\Per_n(0)$ along their shared open set $\cE_n\setminus S\cong \Per_n(0)^{\circ}$ to obtain a projective completion of $\Per_n(0)$, on which $f_{\star}$ is a smooth $\QQ$-rational point.

\end{remark}
\bibliographystyle{alpha} 
\bibliography{BigRefs.bib}
\end{document}